\documentclass[11pt]{article}

\usepackage{amsmath,amsfonts,amssymb,amsthm,graphics,amscd}
\usepackage{latexsym,mathrsfs}

\newtheorem{theorem}{Theorem}[section]
\newtheorem{lemma}[theorem]{Lemma}
\newtheorem{corollary}[theorem]{Corollary}

\newtheorem{example}[theorem]{Example}

\newtheorem{proposition}[theorem]{Proposition}
\newtheorem{remark}[theorem]{Remark}
\newtheorem{question}[theorem]{Question}

\def\NN{\mathbb{N}}
\def\CC{\mathbb{C}}

\def\DD{\mathbb{D}}

\numberwithin{equation}{section}

\def\int{{\rm int}}
\def\asc{{\rm asc}}
\def\dsc{{\rm dsc}}
\def\acc{{\rm acc}}
\def\iso{{\rm iso}}
\def\dim{{\rm dim}}

\def\ind{{\rm ind}}
\def\snoi{\smallskip\noindent}
\def\i{{\rm ind}}
\def\hd{\hat{\delta}}
\def\codim{{\rm codim}\, }
\def\NN{{\mathbb N}}

\def\CC{{\mathbb C}}

\def\DD{\mathbb{D}}
\def\W{\mathcal{W}}
\def\B{\mathcal{B}}
\def\M{\mathcal{M}}
\def\Q{\mathcal{Q}}

\title{Generalized Kato decomposition and essential spectra}

\author{Milo\v s D. Cvetkovi\'c, Sne\v zana \v C. \v Zivkovi\'c-Zlatanovi\'c\footnote{The authors are
supported by the Ministry of Education, Science and Technological
Development, Republic of Serbia, grant no. 174007.}}

\date{}

\begin{document}
\maketitle

\setcounter{page}{1}

\begin{abstract}
\noindent Let ${\bf R}$ denote any of the following classes: upper
(lower) semi-Fredholm operators, Fredholm operators, upper (lower)
semi-Weyl operators, Weyl operators, upper (lower) semi-Browder
operators, Browder operators. For a bounded linear operator $T$ on a
Banach space $X$ we prove that $T=T_M\oplus T_N$ with $T_M \in {\bf
R}$ and $T_N$ quasinilpotent (nilpotent) if and only if $T$ admits a
generalized Kato decomposition ($T$ is of Kato type) and $0$ is not
an interior point of the corresponding spectrum $\sigma_{\bf
R}(T)=\{\lambda \in \CC: T-\lambda \notin {\bf R}\}$. In addition,
we show that every non-isolated boundary point of the spectrum
$\sigma_{\bf R}(T)$ belongs to the generalized Kato spectrum of $T$.
\end{abstract}

2010 {\it Mathematics subject classification\/}:  47A53, 47A10.

{\it Key words and phrases\/}: generalized Kato decomposition,
generalized Kato spectrum, upper and lower semi-Fredholm, semi-Weyl
and semi-Browder operators.

\section{Introduction}

Given a Banach space $X$ and an operator $T \in L(X)$, $T$ is Drazin
(generalized Drazin) invertible if and only if $T$ is the direct sum
of an invertible and a nilpotent (quasinilpotent) operator \cite{K,
lay, koliha}. M. Berkani proved that $T$ is B-Fredholm if and only
if $T=T_1 \oplus T_2$ with $T_1$ Fredholm and $T_2$ nilpotent
\cite[Theorem 2.7]{Ber}. Very recently, pseudo B-Fredholm and pseudo
B-Weyl operators were defined in a sense that $T$ is pseudo
B-Fredholm (pseudo B-Weyl) if $T=T_1 \oplus T_2$, where $T_1$ is
Fredholm (Weyl) and $T_2$ is quasinilpotent \cite{Boasso1,
pseudoBW}. In accordance with these observations it is natural to
study various types of the direct sums. Namely, let ${\bf R}$ denote
any of the following classes: upper (lower) semi-Fredholm operators,
Fredholm operators, upper (lower) semi-Weyl operators, Weyl
operators, upper (lower) semi-Browder operators, Browder operators,
bounded below operators, surjective operators, invertible operators.
The main objective of this article is to provide necessary and
sufficient conditions for an operator $T \in L(X)$ to be the direct
sum of an operator $T_1 \in {\bf R}$ and a quasinilpotent
(nilpotent) operator $T_2$. We approach the problem by using mainly
the gap theory and the advantage of our method lies in the fact that
it enables us to investigate all cases (for all ${\bf R}$ mentioned
above) at the same time.

We are also able to derive some recent results in a different way.
Namely, P. Aiena and E. Rosas prove in \cite{aienarosas} that if $T
\in L(X)$ is of Kato type, then $T$ ($T^{\prime}$) has the SVEP at
$0$ if and only if the approximate point (surjective) spectrum of
$T$ does not cluster at $0$. Q. Jiang and H. Zhong extend this
result to the operators that admit a generalized Kato decomposition
\cite{kinezi}. They show that if $T$ admits a GKD, $T$
($T^{\prime}$) has the SVEP at $0$ if and only if $0$ is not an
accumulation point of its approximate point (surjective) spectrum,
and it is exactly when $0$ is not an interior point of the
approximate point (surjective) spectrum of $T$. We extend these
results to the cases of the essential spectra. Namely, we show for
every class ${\bf R}$ mentioned earlier that if $T$ admits a GKD
($T$ is of Kato type), $0$ is not an accumulation point of the
spectrum $\sigma_{\bf R}(T)=\{\lambda \in \CC: T-\lambda \notin {\bf
R}\}$ if and only if $0$ is not an interior point of $\sigma_{\bf
R}(T)$. In particular, if ${\bf R} \in \{\text{bounded below
operators, surjective operators}\}$ we obtain the result of Q. Jiang
and H. Zhong. What is more, for $T \in L(X)$ it is proved that if
$0$ is a boundary point of $\sigma_{{\bf R}}(T)$, $T$ admits a GKD
($T$ is of Kato type) if and only if $T$ is the direct sum of an
operator that belongs to the class ${\bf R}$ and a quasinilpotent
(nilpotent) operator. In the special case of invertible operators
from this we derive the result of P. Aiena and E. Rosas (the result
of Q. Jiang and H. Zhong) which says that if $0$ is a boundary point
of the spectrum of $T \in L(X)$, then $T$ is of Kato type ($T$
admits a GKD) if and only if $T$ is Drazin (generalized Drazin)
invertible; see \cite[Theorem 2.9]{aienarosas} and \cite[Theorem
3.8]{kinezi}. It is worth mentioning that this result of P. Aiena
and E. Rosas is an improvement of an earlier result of M. Berkani
which shows that if $0$ is isolated in the spectrum of $T$, then $T$
is a B-Weyl operator if and only if it is Drazin invertible
\cite[Theorem 4.2]{Ber2}.

In Section 2 we set up terminology and recall necessary facts. Our
main results are established in Section 3. Given an operator $T \in
L(X)$, $X$ is a Banach space, we prove that $T=T_M\oplus T_N$ with
$T_M \in {\bf R}$ and $T_N$ quasinilpotent if and only if $T$ admits
a generalized Kato decomposition and $0$ is not an interior point of
$\sigma_{\bf R}(T)$. What is more, the condition that $T$ admits a
generalized Kato decomposition cannot be omitted and it will be
demonstrated by an example. We also show that necessary and
sufficient for $T \in L(X)$ to be the direct sum $T=T_M \oplus T_N$
with $T_M \in {\bf R}$ and $T_N$ quasinilpotent is that $T$ admits a
generalized Kato decomposition and that $0$ is not an accumulation
point of the spectrum $\sigma_{\bf R}(T)$. If ${\bf R}$ is the class
of invertible operators the condition that $T$ admits a GKD is
unnecessary \cite[Theorem 7.1]{koliha}. If $\sigma_{\bf R}(T)$ is
any of the essential spectra we demonstrate by an example that this
condition can not be excluded. If ${\bf R}$ represents the class of
bounded below or surjective operators, we are not sure whether this
condition can be omitted and it is an open question. Also, it is
possible to obtain some results related to the B-Fredholm and B-Weyl
operators. In particular, an operator $T$ is B-Fredholm (B-Weyl) if
and only if it is of Kato type and if $0$ is not an interior point
of its Fredholm (Weyl spectrum).

In the fourth section some applications are presented. We prove that
every boundary point of $\sigma_{\bf R}(T)$, where ${\bf R}$ is any
of the classes mentioned above, which is also an accumulation point
of $\sigma_{\bf R}(T)$ belongs to the generalized Kato spectrum. In
particular, if ${\bf R}$ is the class of invertible operators we
obtain \cite[Corollary 3.6]{kinezi}. If $T \in L(X)$, let
$\sigma_{{\bf gDR}}(T)$ be the set of all $\lambda \in \CC$ such
that $T-\lambda$ can not be represented as the direct sum of an
operator from the class ${\bf R}$ and a quasinilpotent operator. We
show that the connected hull of the spectrum $\sigma_{{\bf gDR}}(T)$
coincide with the connected hull of the generalized Kato spectrum
for every class ${\bf R}$. In particular, the connected hulls of the
generalized Drazin spectrum and  the generalized Kato spectrum are
equal. Moreover, the connected hulls of the B-Fredholm, B-Weyl,
Drazin and of the Kato type spectrum are equal.

\section{Preliminaries}
Let $\mathbb{N} \, (\mathbb{N}_0)$ denote the set of all positive
(non-negative) integers, and let $\mathbb{C}$ denote the set of all
complex numbers. Let $X$ be an infinite dimensional Banach space and
let $L(X)$ be the Banach algebra of all bounded linear operators
acting on $X$. Given $T \in L(X)$, we denote by $N(T), \, R(T)$ and
$\sigma (T)$, the {\em kernel}, the {\em range} and the {\em
spectrum} of $T$, respectively. In addition, $\alpha(T)$ and
$\beta(T)$ will stand for {\em nullity} and {\em defect} of $T$. The
space of bounded linear functionals on $X$ is denoted by
$X^{\prime}$. If $K \subset \mathbb{C}$, then $\partial K$ is the
boundary of $K$, $\acc \, K$ is the set of accumulation points of
$K$, $\iso \, K=K \setminus \acc \, K$ and $\int\, K$ is the set of
interior points of $K$. The group of all invertible operators is
denoted by $L(X)^{-1}$.

Recall that $T$ is said to be {\em nilpotent} when $T^n=0$ for some
$n \in \mathbb{N}$, while $T$ is {\em quasinilpotent} if
$\|T^n\|^{1/n}\to 0$, that is $T-\lambda \in L(X)^{-1}$  for all
complex $\lambda\ne 0$. An operator $T \in L(X)$ is {\em bounded
below} if there exists some $c>0$ such that $c\|x\|\leq \|Tx\| \; \;
\text{for every} \; \; x \in X$. Let $\M(X)$ denote the set of all
bounded below operators, and let $\Q(X)$ denote the set of all
surjective operators. The approximate point spectrum of $T\in L(X)$
is defined by
$$
\sigma_{ap}(T)=\{\lambda\in\CC: T-\lambda\ {\rm is\ not\ bounded\
below}\}
$$
and the surjective spectrum is defined by
$$
\sigma_{su}(T)=\{\lambda\in\CC: T-\lambda\ {\rm is\ not\
surjective}\}.
$$
 An operator $T \in
L(X)$ is {\em Kato} if $R(T)$ is closed and $N(T) \subset R(T^n), \,
n \in \mathbb{N}_0$. If $R(T)$ is closed and $\alpha(T)<\infty$,
then $T \in L(X)$ is said to be {\em upper semi-Fredholm}. An
operator $T \in L(X)$ is {\em lower semi-Fredholm} if
$\beta(T)<\infty$. The set of upper semi-Fredholm operators (lower
semi-Fredholm operators) is denoted by $\Phi_+(X)$ ($\Phi_-(X))$. If
$T$ is upper or lower semi-Fredholm operator then the index of $T$
is defined as $\ind(T)=\alpha(T)-\beta(T)$. An operator $T$ is {\em
Fredholm} if both $\alpha(T)$ and $\beta(T)$ are finite. We will
denote by $\Phi(X)$ the set of Fredholm operators. The sets of {\em
upper semi-Weyl}, {\em lower semi Weyl} and {\em Weyl} operators are
defined by $\W_+(X)=\{T \in \Phi_+(X):\ind(T)\leq0\}$, $\W_-(X)=\{T
\in \Phi_-(X):\ind(T)\geq0\}$ and $\W(X)=\{T \in
\Phi(X):\ind(T)=0\}$, respectively. B-Fredholm and B-Weyl operators
were introduced and studied by M. Berkani \cite{Ber, Ber0, Ber2}. An
operator $T \in L(X)$ is said to be B-Fredholm (B-Weyl) if there is
$n \in \mathbb{N}$ such that $R(T^n)$ is closed and the restriction
$T_{n} \in L(R(T^n))$ of $T$ to $R(T^n)$ is Fredholm (Weyl). The
B-Fredholm and the B-Weyl spectrum of $T$ are defined by
\begin{align*}
\sigma_{B\Phi}(T) &=\{\lambda \in \CC: T-\lambda \; \text{is not
B-Fredholm}\}, \\
\sigma_{B\W}(T) &=\{\lambda \in \CC: T-\lambda \; \text{is not
B-Weyl}\}, \; \text{respectively}. \end{align*} \noindent Recall
that $T \in L(X)$ is said to be {\em Riesz operator}, if $T-\lambda
\in \Phi(X)$ for every non-zero $\lambda \in \mathbb{C}$.

The {\em ascent} of $T$ is defined as $\asc(T)=\inf\{n \in
\mathbb{N}_0:N(T^n)=N(T^{n+1})\}$, and {\em descent} of $T$ is
defined as $\dsc(T)=\inf\{n \in \mathbb{N}_0:R(T^n)=R(T^{n+1})\}$,
where the infimum over the empty set is taken to be infinity. An
operator $T \in L(X)$ is {\em upper semi-Browder} if $T$ is upper
semi-Fredholm and $\asc(T) < \infty$. If $T \in L(X)$ is lower
semi-Fredholm and $\dsc(T)< \infty$, then $T$ is {\em lower
semi-Browder}. Let $\mathcal{B}_+(X)$ ($\mathcal{B}_-(X)$) denote
the set of all upper (lower) semi-Browder operators. The set of
Browder operators is defined by $\mathcal{B}(X)=\mathcal{B}_+(X)
\cap \mathcal{B}_-(X)$.

If $M$ is a subspace of $X$ such that $T(M) \subset M$, $T \in
L(X)$, it is said that $M$ is {\em $T$-invariant}. We define $T_M:M
\to M$ as $T_Mx=Tx, \, x \in M$.  If $M$ and $N$ are two closed
$T$-invariant subspaces of $X$ such that $X=M \oplus N$, we say that
$T$ is {\em completely reduced} by the pair $(M,N)$ and it is
denoted by $(M,N) \in Red(T)$. In this case we write $T=T_M \oplus
T_N$ and say that $T$ is the {\em direct sum} of $T_M$ and $T_N$.

An operator $T \in L(X)$ is said to admit a {\em generalized Kato
decomposition}, abbreviated as GKD, if there exists a pair $(M,N)
\in Red(T)$ such that $T_M$ is Kato and $T_N$ is quasinilpotent. A
relevant case is obtained if we assume that $T_N$ is nilpotent. In
this case $T$ is said to be of {\em Kato type}. An operator is said
to be {\em essentially Kato} if it admits a GKD $(M,N)$ such that
$N$ is finite-dimensional. If $T$ is essentially Kato then $T_N$ is
nilpotent, since every quasinilpotent operator on a finite
dimensional space is nilpotent. The classes $\Phi_+(X), \,
\Phi_-(X), \, \Phi(X), \, \mathcal{B}_+(X), \, \mathcal{B}_-(X)$,
$\mathcal{B}(X)$, $\mathcal{W}_+(X)$, $\mathcal{W}_-(X)$ and
$\mathcal{W}(X)$ belong to the class of essentially Kato operators
\cite[Theorem 16.21]{Mu}. For $T\in L(X)$, {\em the Kato spectrum},
{\em the  essentially Kato spectrum}, { \em the Kato type spectrum}
and {\it the generalized Kato spectrum} are defined by

\begin{align*} \sigma_K(T) &=\{\lambda \in \CC: T-\lambda \ \text{ is\ not\
Kato}\}, \\ \sigma_{eK}(T) &=\{\lambda \in \CC: T-\lambda \ \text{
is\ not\ essentially\ Kato}\}, \\ \sigma_{Kt}(T) &=\{\lambda \in
\CC: T-\lambda\ \text{ is\ not\ of\ Kato\ type}\}, \\ \sigma_{gK}(T)
&= \{\lambda \in \CC: T-\lambda\ \text{ does\ not\ admit\ a\ GKD}\},
\end{align*} respectively. Clearly,
\begin{equation}\label{poc.inkluzija}
\sigma_{gK}(T)\subset \sigma_{Kt}(T)\subset
\sigma_{eK}(T)\subset\sigma_{K}(T)\subset \sigma_{ap}(T)\cap
\sigma_{su}(T).
\end{equation}

The {\em quasinilpotent part} $H_0(T)$ of an operator $T \in L(X)$
is defined by
\[H_0(T)=\{x \in X: \lim_{n \to +\infty} \|T^nx\|^{1/n}=0\}.\]
It is easy to verify that $H_0(T)=\{0\}$ if $T$ is bounded below. An
operator $T \in L(X)$ is quasinilpotent if and only if $H_0(T)=X$
\cite[Theorem 1.68]{aiena2}.

The {\em analytical core} of $T$, denoted by $K(T)$, is the set of
all $x \in X$ for which there exist $c>0$ and a sequence $(x_n)_n$
in $X$ satisfying
\[Tx_1=x, \; \; Tx_{n+1}=x_n \; \; \text{for all} \; \; n \in \NN,
\; \; \|x_n\|\leq c^n\|x\| \; \; \text{for all} \; \; n \in \NN.\]
If $T$ is surjective, then $K(T)=X$ \cite[Theorem 1.22]{aiena2}.

An operator $T\in L(X)$ is said to be {\em generalized Drazin
invertible}, if there exists $B \in L(X)$ such that
\[TB=BT, \; \; \; BTB=B, \; \; \; TBT-T \; \; \text{is quasinilpotent}.\]
\noindent The generalized Drazin spectrum of $T\in L(X)$ is defined
by
$$
\sigma_{gD}(T)=\{\lambda\in\CC: T-\lambda\ {\rm is\ not\
generalized\ Drazin\ invertible}\}.
$$

\noindent The equivalent conditions to the existence of generalized
Drazin inverse of a bounded  operator are collected in the following
theorem.
\begin{theorem}[see \cite{koliha}, \cite{M1}, \cite{schmoeger}, \cite{koliha2}, \cite{DajicKoliha}] \label{Kol-Dra} Let $T\in L(X)$.
The following conditions are equivalent:

\snoi {\rm (i)} $T$ is generalized Drazin invertible;

\snoi {\rm (ii)} There exists a bounded projection $P$ on $X$ which
commutes with $T$ such that $T+P$ is invertible and $TP$ is
quasinilpotent;

\snoi {\rm (iii)} $0\notin \acc\, \sigma (T)$;

\snoi {\rm (iv)} There is a bounded projection $P$ on $X$ such that
$R(P)=H_0(T)$ and $N(T)=K(T)$;

\snoi {\rm (v)} There exists $(M,N)\in Red(T)$ such that $ T_M$ is
invertible and $ T_N$ is quasinilpotent;

\snoi {\rm (vi)}   $X=K(T)\oplus H_0(T)$ with at least one of the
component spaces closed.
\end{theorem}

For a subspace $M$ of $X$ its annihilator $M^\bot$ is defined by
$$
M^\bot=\{f\in X^\prime:f(x)=0\ {\rm for\ all\ } x\in M\}.
$$
Recall that if $M$ is closed, then
\begin{equation}\label{gp0}
 \dim M^\bot=\codim M.
\end{equation}
Let $M$ and $L$ be two  subspaces of $X$ and let
\[\delta(M,L)=\sup \{dist(u,L): u \in M, \, \|u\|=1\},\]
in the case that $M \neq \{0\}$, otherwise we define $\delta(\{0\},
L)=0$ for any subspace $L$. The {\em gap} between $M$ and $L$ is
defined by
\[\hat{\delta}(M,L)=\max\{\delta(M,L), \delta(L,M)\}. \]
It is known that \cite[corollary 10.10]{Mu}
\begin{equation}\label{gp1}
  \hat{\delta}(M,L)<1\Longrightarrow\dim M=\dim L.
\end{equation}
If $M$ and $L$ are closed subspaces of $X$, then  \cite[Theorem
10.8]{Mu}
\begin{equation}\label{gp2}
\hat{\delta}(M^\bot,L^\bot)= \hat{\delta}(M,L).
\end{equation}
Therefore, for closed subspaces $M$ and $L$  of $X$, according to
(\ref{gp0}), (\ref{gp1}) and (\ref{gp2}), there is implication
\begin{equation}\label{gp3}
  \hat{\delta}(M,L)<1\Longrightarrow\codim M=\codim L.
\end{equation}

We use the following notation.
\begin{center}
\begin{tabular}{|c|c|c|} \hline
${\bf R_1}=\Phi_+(X)$ & ${\bf R_2}=\Phi_-(X)$ & ${\bf R_3}=\Phi(X)$ \\
\hline
${\bf R_4}=\mathcal{W}_+(X)$ & ${\bf R_5}=\mathcal{W}_-(X)$ & ${\bf R_6}=\mathcal{W}(X)$ \\
\hline
${\bf R_7}=\mathcal{B}_+(X)$ & ${\bf R_8}=\mathcal{B}_-(X)$ & ${\bf R_9}=\mathcal{B}(X)$ \\
\hline
${\bf R_{10}}=\M(X)$ & ${\bf R_{11}}=\Q(X)$ & ${\bf R_{12}}=L(X)^{-1}$ \\
\hline
\end{tabular}
\end{center}
\noindent The sets ${\bf R}_i$, $1\leq i \leq 12$, are open in
$L(X)$ and contain $L(X)^{-1}$ (for the openness of the set of upper
(lower) semi-Browder operators see \cite[Satz 4]{KV}). The spectra
$\sigma_{{\bf R}_i}(T)=\{\lambda \in \mathbb{C}: T-\lambda \not \in
{\bf R}_i\}$, $1\leq i \leq 12$, are non-empty and compact subsets
of $\mathbb{C}$. We write $\sigma_{{\bf
R}_1}(T)=\sigma_{\Phi_+}(T)$, $\sigma_{{\bf
R}_2}(T)=\sigma_{\Phi_-}(T)$, etc, and
$\rho_{\Phi_+}(T)=\CC\setminus \sigma_{\Phi_+}(T)$,
$\rho_{\Phi_-}(T)=\CC\setminus \sigma_{\Phi_-}(T)$, etc. In
particular, $\sigma_{{\bf R}_{10}}(T)=\sigma_{ap}(T)$, $\sigma_{{\bf
R}_{11}}(T)=\sigma_{su}(T)$, $\rho_{ap}(T)=\CC\setminus
\sigma_{ap}(T)$ and $\rho_{su}(T)=\CC\setminus \sigma_{su}(T)$. We
consider the following classes of bounded linear operators:
\[{\bf g D R}_i=\left\{ T \in L(X) : \begin{array}{c} \text{there
exists} \, (M,N)\in Red(T) \; \text{such that} \\ T_M \in {\bf R}_i
\; \text{and} \; T_N \; \text{is quasinilpotent} \end{array}
\right\}, \; \; \; 1\leq i \leq 12.\] If $T_N$ mentioned in this
definition is nilpotent then it is said that $T$ belongs to the
class ${\bf DR}_i$, $1 \leq i \leq 12$. It is clear that ${\bf R}_i
\subset {\bf DR}_i \subset {\bf g D R}_i, 1 \leq i \leq 12$.

We shall say that $T\in L(X)$  is {\it generalized Drazin upper
semi-Fredholm } (resp. {\it generalized Drazin  lower semi-Fredholm,
generalized Drazin Fredholm,       generalized Drazin  upper
semi-Weyl, generalized Drazin  lower semi-Weyl, generalized Drazin
Weyl, generalized Drazin bounded below, generalized Drazin
surjective}) if $T\in {\bf g D \Phi_+}(X) $ (resp. ${\bf g D
\Phi_-}(X)$, ${\bf g D \Phi}(X)$, ${\bf g D \W_+}(X)$, ${\bf g D
\W_-}(X)$,  ${\bf g D \W}(X)$,  ${\bf g D \M}(X)$, $ {\bf g D
\Q}(X)$). The reason for introducing these names is that all these
classes generalize the class of generalized Drazin invertible
operators and, as we will see, may be characterized in a similar way
as the class of generalized Drazin invertible operators. We remark
that pseudo B-Fredholm operators and generalized Drazin Fredholm
operators coincide, as well as,  pseudo B-Weyl operators and
generalized Drazin Weyl operators.

The following technical lemma will be useful in the sequel.
\begin{lemma} \label{lema}
Let $T \in L(X)$ and $(M,N) \in Red(T)$. The following statements
hold:
\par \noindent {\rm (i)} $T \in {\bf R}_i$ if and only if $T_M
\in {\bf R}_i$ and $T_N \in {\bf R}_i$, $1\le i\le 3$ or $7\le i\le
12$, and in that case $\ind (T)=\ind(T_M)+\ind(T_N)$;\par \noindent
{\rm (ii)} If $T_M \in {\bf R}_i$ and $T_N \in {\bf R}_i$, then $T
\in {\bf R}_i$, $4 \leq i \leq 6$.

\snoi{\rm (iii)} If $T \in {\bf R}_i$ and $T_N$ is Weyl, then $T_M
\in {\bf R}_i$, $ 4\leq i \leq 6$.
\end{lemma}
\begin{proof}
{\rm (i)}: From the equalities $N(T)=N(T_M)\oplus N(T_N)$ and
$R(T)=R(T_M)\oplus R(T_N)$ it follows that
$\alpha(T)=\alpha(T_M)+\alpha(T_N)$ and
$\beta(T)=\beta(T_M)+\beta(T_N)$. It implies that
  $\alpha(T)<\infty$ if and only if
$\alpha(T_M)<\infty$ and $\alpha(T_N)<\infty$, and also,
$\beta(T)<\infty$ if and only if $\beta(T_M)<\infty$ and
$\beta(T_N)<\infty$. It is known that $R(T)$ is closed if and only
if $R(T_M)$ and $R(T_N)$ are closed \cite[Lemma 3.3]{kinezi}.
Therefore $T$ is bounded below (surjective, upper semi-Fredholm,
lower semi-Fredholm) if and only if $T_M$ and $T_N$ are bounded
below (surjective, upper semi-Fredholm, lower semi-Fredholm), and in
that case $\ind (T)=\alpha (T)-\beta
(T)=(\alpha(T_M)+\alpha(T_N))-(\beta(T_M)+\beta(T_N))=\ind(T_M)+\ind(T_M)$.

Since $N(T^n)=N(T_M^n)\oplus N(T_N^n)$, for every $n\in \NN$, we
conclude that $\asc(T)<\infty$ if and only if $\asc(T_M)<\infty$ and
$\asc(T_N)<\infty$, with $\asc(T)={\rm max}\{\asc(T_M),\asc(T_N)\}$.
Similarly,  as  $R(T^n)=R(T_M^n)\oplus R(T_N^n)$, $n\in\NN$, we get
that  $\dsc(T)<\infty$ if and only if $\dsc(T_M)<\infty$ and
$\dsc(T_N)<\infty$, with $\dsc(T)={\rm max}\{\dsc(T_M),\dsc(T_N)\}$.

\smallskip

{\rm (ii)}: Follows from {\rm (i)}.

{\rm (iii)}: Suppose that $T\in\W_+(X)$ and that $T_N$ is Weyl.
According to {\rm (i)} it follows that $T_M\in\Phi_+(X)$ and $\ind
(T_M)=\ind(T_M)+\ind(T_N)=\ind (T)\le 0$. Thus $T_M$ is upper
semi-Weyl. The cases $i=5$ and $i=6$ can be proved similarly.
\end{proof}

\section{Main results}

We start with the result which is essential for our work.

\begin{proposition} \label{gap0}
Let $T \in L(X)$. Then the following implications hold:
\par \noindent {\rm (i)} If $T$ is Kato and $0\in {\rm acc\, }\rho_{\Phi_+}(T)$ ($0\in {\rm acc\, }\rho_{\Phi_-}(T)$), then $T$ is upper (lower) semi-Fredholm;

\par \noindent {\rm (ii)} If $T$ is Kato and $0\in {\rm acc\, }\rho_{\W_+}(T)$ ($0\in {\rm acc\, }\rho_{\W_-}(T)$), then $T$ is upper (lower) semi-Weyl;

\par \noindent {\rm (iii)} If $T$ is Kato and $0\in {\rm acc\, }\rho_{ap}(T)$ ($0\in {\rm acc\, }\rho_{su}(T)$), then $T$ is bounded below (surjective);

\par \noindent {\rm (iv)} If $T$ is Kato and $0\in {\rm acc\, }\rho_{\B_+}(T)$ ($0\in {\rm acc\, }\rho_{\B_-}(T)$), then $T$ is bounded below (surjective).
\end{proposition}
\begin{proof} {\rm (i)}:  Suppose that $T$ is Kato. According to \cite[Corollary 12.4]{Mu}, $T-\lambda$ is Kato for all $\lambda$ in a neighborhood of $0$. From \cite[Theorem 12.2]{Mu} it follows that $\lim\limits_{\lambda\to 0}\hd (N(T),N(T-\lambda))=0$ and $\lim\limits_{\lambda\to 0}\hd (R(T-\lambda),R(T))=0$ and hence, there exists  $\delta>0$ such that for  all $|\lambda|<\delta$ it holds that
 \begin{equation}\label{gap-0i}
 \hd (N(T),N(T-\lambda))<1,\  \ \hd (R(T-\lambda),R(T))<1,
 \end{equation}
 \begin{equation}\label{gap-00i}
   T-\lambda\ {\rm is\ Kato},
 \end{equation}
 ${\rm and\ hence\ }R(T-\lambda)\ {\rm is\ closed}$.  By \cite[Corollary 10.10]{Mu}, (\ref{gap-0i}), (\ref{gap-00i}) and (\ref{gp3}), for all $|\lambda|<\delta$  the following holds
\begin{equation}\label{gap-i}
  \dim N(T-\lambda)=\dim N(T)\  \ {\rm i}\ \  \codim R(T)=\codim R(T-\lambda).
\end{equation}
 Suppose now that $0\in {\rm acc\, }\rho_{\Phi_+}(T)$ ($0\in {\rm acc\, }\rho_{\Phi_-}(T)$). Then there exists $\mu$ such that $0<|\mu|<\delta$ and $T-\mu$ is upper semi-Fredholm (lower semi-Fredholm), and so from (\ref{gap-i}) it follows that $\dim N(T)=\dim N(T-\mu)<+\infty$ ($\codim R(T)=\codim  R(T-\mu)<+\infty$). Consequently, $T$ is upper semi-Fredholm (lower semi-Fredholm).

\snoi  {\rm (ii)}: Suppose that $T$ is Kato and $0\in {\rm acc\,
}\rho_{\W_+}(T)$. Then there exists  $\delta>0$  such that
(\ref{gap-00i}) and  (\ref{gap-i}) hold for all $|\lambda|<\delta$.
Since $0\in {\rm acc\, }\rho_{\W_+}(T)$ , there exists $\mu$ such
that $0<|\mu|<\delta$ and $T-\mu$ is upper semi-Weyl.  From
(\ref{gap-i}) it follows that $\dim N(T)=\dim N(T-\mu)<\infty$,
$\codim R(T)=\codim  R(T-\mu)$ and $\i(T)=\dim N(T)-\codim R(T)=\dim
N(T-\mu)-\codim R(T-\mu)=\i(T-\mu)\le 0$ and so $T$ is upper
semi-Weyl. The proof of $T$ being lower semi-Weyl if $T$ is Kato and
$0\in {\rm acc\, }\rho_{\W_-}(T)$ is analogous.

\snoi  {\rm (iii)}: Let $T$ be Kato and let $0\in {\rm acc\,
}\rho_{ap}(T)$ ($0\in {\rm acc\, }\rho_{su}(T)$). Dealing as in the
previous part of the proof we find $\delta>0$ and $\mu\in\CC$,
$0<|\mu|<\delta$, such that
 (\ref{gap-00i}) and  (\ref{gap-i}) hold for all $|\lambda|<\delta$   and $T-\mu$ is bounded below (surjective).
From (\ref{gap-i}) it follows that $\dim N(T)=\dim N(T-\mu)=0$
($\codim R(T)=\codim  R(T-\mu)=0$), and so $T$ is bounded below
(surjective).

\snoi  {\rm (iv)}: Let $T$ be Kato and let $0\in {\rm acc\,
}\rho_{\B_+}(T)$ ($0\in {\rm acc\, }\rho_{\B_-}(T)$). Then there
exists  $\delta>0$  such that
 (\ref{gap-00i}) and  (\ref{gap-i}) hold for all $|\lambda|<\delta$. From $0\in {\rm acc\, }\rho_{\B_+}(T)$ ($0\in {\rm acc\, }\rho_{\B_-}(T)$) it follows that there exists $\mu$ such that $0<|\mu|<\delta$ and $T-\mu$ is upper semi-Browder (lower  semi-Browder). Since $T-\mu$ is Kato, from \cite[Lemma 20.9]{Mu} it follows that $T-\mu $  is bounded below (surjective). Now, as in the previous part of the proof, we can conclude that $T$ is bounded below (surjective).
\end{proof}

\begin{proposition} \label{open}
Let $T \in L(X)$ and $1\leq i \leq 12$. If $T$ belongs to the set
 ${\bf g D R}_i$, then $0 \not \in \acc \, \sigma_{{\bf R}_i}(T)$.
\end{proposition}
\begin{proof}
Let $(M,N) \in Red(T)$ such that $T_M \in {\bf R}_i$ and $T_N$ is
quasinilpotent. Since ${\bf R}_i$ is open,  there exists
$\epsilon>0$ such that $(T-\lambda)_M=T_M-\lambda \in {\bf R}_i$ for
$|\lambda|<\epsilon$. On the other hand, $(T-\lambda)_N=T_N-\lambda
\in L(X)^{-1} \subset {\bf R}_i$ for every $\lambda \neq 0$. Now by
applying   Lemma ďż˝\ref{lema} we obtain that  $T-\lambda \in
{\bf R}_i$ for $0<|\lambda|<\epsilon$, and so $0 \not \in \acc \,
\sigma_{{\bf R}_i}(T)$.
\end{proof}
We now state the first main result.

\begin{theorem} \label{glavna}
Let $T \in L(X)$ and $1 \leq i \leq 6$. The following conditions are
equivalent:\par \noindent {\rm (i)} There exists $(M,N)\in Red(T)$
such that $ T_M\in {\bf R}_i$  and $ T_N$ is quasinilpotent, that is
 $T \in {\bf g D R}_i$;\par
\noindent {\rm (ii)} $T$ admits a GKD and $0 \not \in \acc \,
\sigma_{{\bf R}_i}(T)$;
\par\noindent {\rm (iii)} $T$ admits a GKD and
$0  \notin \int \, \sigma_{{\bf R}_i}(T)$;

\par \noindent {\rm
(iv)} There exists a projection $P \in L(X)$ that commutes with $T$
such that $T+P \in {\bf R}_i$ and $TP$ is quasinilpotent.
\end{theorem}
\begin{proof}
{\rm (i)} $\Longrightarrow$ {\rm (ii)}:  Let   $T=T_M \oplus T_N$,
where $T_M\in {\bf R}_i$ and $T_N$ is quasinilpotent. Then $0 \not
\in \acc \, \sigma_{{\bf R}_i}(T)$ by Proposition  ďż˝\ref{open}.
From   \cite[Theorem 16.21]{Mu} it follows that there exist two
closed $T$-invariant subspaces $M_1$ and $M_2$ such that $M=M_1
\oplus M_2, \, M_2 \, \text{is finite dimensional}$, $T_{M_1}$ is
Kato and $T_{M_2}$ is nilpotent. We have $X=M_1 \oplus (M_2 \oplus
N)$, $M_2 \oplus N$ is closed, $T_{M_2 \oplus N}=T_{M_2} \oplus T_N$
is quasinilpotent and thus $T$ admits the GKD $(M_1,M_2 \oplus N)$.

{\rm (ii)} $\Longrightarrow$ {\rm (iii)}: Clear.
\smallskip

{\rm (iii)} $\Longrightarrow$ {\rm (i)}: Let $i\in\{1,2,3\}$. Assume
that $T$ admits a GKD and $0\notin\int\, \sigma_{{\bf R}_i}(T)$,
that is $0  \in \acc \, \rho_{{\bf R}_i}(T)$. Then  there exists
$(M,N) \in Red(T)$ such that $T_M$ is Kato and $T_N$ is
quasinilpotent, and also, because of $0  \in \acc \, \rho_{{\bf
R}_i}(T)$, according to Lema \ref{lema}{\rm (i)}, it follows that $0
\in \acc \, \rho_{{\bf R}_i}(T_M)$. From Proposition \ref{gap0}{\rm
(i)}  it follows that $T_M\in {\bf R}_i$, and so $T \in {\bf g D
R}_i(X)$.

Suppose that $T$ admits a GKD and $0\notin\int\, \sigma_{\W_+}(T)$,
i.e. $0  \in \acc \, \rho_{\W_+}(T)$. Then  there exists $(M,N) \in
Red(T)$ such that $T_M$ is Kato and $T_N$ is quasinilpotent. We show
that $0  \in \acc \, \rho_{\W_+}(T_M)$. Let $\epsilon>0$. From $0
\in \acc \, \rho_{\W_+}(T)$ it follows that there exists
$\lambda\in\CC$ such that $0<|\lambda|<\epsilon$  and
$T-\lambda\in\W_+(X)$. As $T_N$ is quasinilpotent, $T_N- \lambda$ is
invertible, and so, according to Lema \ref{lema}{\rm (iii)}, we
conclude that $T_M-\lambda\in\W_+(M)$, that is $\lambda\in
\rho_{\W_+}(T_M)$. Therefore, $0  \in \acc \, \rho_{\W_+}(T_M)$ and
from Proposition \ref{gap0}{\rm (ii)}  it follows that $T_M$
 is upper semi-Weyl,  and so $T \in {\bf g D \W_+}(X)$. The cases $i=5$ and $i=6$ can be proved similarly.

\par
\smallskip

 {\rm (i)} $\Longrightarrow$ {\rm (iv)}: Suppose that there exists
$(M,N) \in Red(T)$ such that $T_M\in {\bf R}_i$  and $T_N$ is
quasinilpotent. Let $P \in L(X)$ be a projection such that $N(P)=M$
and $R(P)=N$.  Then $TP=PT$ and every element $x \in X$ may be
represented as $x=x_1+x_2$, where $x_1 \in M$ and $x_2 \in N$. Also,
\[\|(TP)^nx\|^{\frac{1}{n}}=\|T^nPx\|^{\frac{1}{n}}=\|(T_N)^nx_2\|^{\frac{1}{n}} \to 0 \; (n \to \infty),\]
since $T_N$ is quasinilpotent. We obtain $H_0(TP)=X$, so $TP$ is
quasinilpotent. Since $(T+P)_M=T_M $ and $(T+P)_N=T_N+I_N \in
L(N)^{-1}$, where $I_N$ is identity on $N$, we have  that
$(T+P)_M\in {\bf R}_i$ and $(T+P)_N\in {\bf R}_i$ and hence, $T+P\in
{\bf R}_i $ by Lemma ďż˝\ref{lema}{\rm (i)} and {\rm (ii)}.

\smallskip

{\rm (iv)} $\Longrightarrow$ {\rm (i)}:  Assume  that there exists a
projection $P \in L(X)$ that commutes with $T$ such that $T+P \in
{\bf R}_i$  and $TP$ is quasinilpotent. Put $N(P)=M$ and $R(P)=N$.
Then $X=M \oplus N$, $T(M) \subset M$ and $T(N) \subset N$. For
every $x \in N$ we have
\[\|(T_N)^nx\|^{\frac{1}{n}}=\|T^nP^nx\|^{\frac{1}{n}}=\|(TP)^nx\|^{\frac{1}{n}} \to 0 \; (n \to
\infty), \] since $TP$ is quasinilpotent. It follows that
$H_0(T_N)=N$ and hence, $T_N$ is quasinilpotent. It remains to prove
that $T_M\in {\bf R}_i$. For $i\in\{1,2,3\}$, by Lemma
ďż˝\ref{lema}{\rm (i)} we deduce that  $T_M=(T+P)_M\in {\bf R}_i
$. Set $i=4$. Since  $T_N$ is quasinilpotent, it follows that
$T_N+I_N$ is invertible, where $I_N$ is identity on $N$. From $T+P
\in \mathcal{W}_+(X)$ and  the decomposition
\begin{equation*}
T+P=(T+P)_M \oplus (T+P)_N=T_M \oplus (T_N+I_N),
\end{equation*}
according to Lemma ďż˝\ref{lema}{\rm (iii)},  we conclude  that
$T_M \in \mathcal{W}_+(M)$. For $i=5$ and $i=6$ we apply similar
consideration.
\end{proof}
\noindent Q. Jiang and H. Zhong show in \cite{kinezi} that if $T \in
L(X)$ admits a GKD, $\sigma_{ap}(T)$ ($\sigma_{su}(T)$) does not
cluster at $0$ if and only if $0$ is not an interior point of
$\sigma_{ap}(T)$ ($\sigma_{su}(T)$). In Theorems ˜\ref{t1} and
˜\ref{t2} we obtain their result in a different way and also provide
further conditions that are equivalent to the previous two.
\begin{theorem} \label{t1}
Let $T \in L(X)$. The following conditions are equivalent:
\par \noindent {\rm (i)} $H_0(T)$ is
closed and there exists a closed subspace $M$ of $X$ such that
$(M,H_0(T))\in Red(T)$ and $T(M)$ is closed;

\par \noindent {\rm
(ii)} There exists $(M,N)\in Red(T)$ such that $ T_M$ is bounded
below and $ T_N$ is quasinilpotent, that is
 $T\in {\bf g D\mathcal{M}}(X)$;


\par\noindent {\rm (iii)} $T$ admits a GKD and
$0\notin{\rm acc}\, \sigma_{ap}(T)$;

\par\noindent {\rm (iv)} $T$ admits a GKD and
$0\notin{\rm int}\, \sigma_{ap}(T)$;

\par \noindent {\rm (v)}
There exists a bounded projection $P$ on $X$ which commutes with $T$
such that $T+P$ is bounded below and $TP$ is quasinilpotent;

\par \noindent {\rm
(vi)} There exists $(M,N)\in Red(T)$ such that $ T_M$ is upper
semi-Browder and $ T_N$ is quasinilpotent, that is $T\in {\bf g
D{\mathcal{B}_+}}(X)$;

\par \noindent {\rm (vii)} $T$ admits a GKD  and
$0\notin {\rm acc}\, \sigma_{\mathcal{B}_+}(T)$;

\par \noindent {\rm (viii)} $T$ admits a GKD  and
$0\notin {\rm int}\, \sigma_{\mathcal{B}_+}(T)$;

\par \noindent {\rm (ix)}
There exists a bounded projection $P$ on $X$ which commutes with $T$
such that $T+P$ is upper semi-Browder and $TP$ is quasinilpotent.

\par In particular, if $T$ satisfies any of the conditions {\rm (i)}-{\rm (ix)},
then the subspace $N$ in {\rm (ii)} is uniquely determined and
$N=H_0(T)$.
\end{theorem}
\begin{proof} {\rm (i)} $\Longrightarrow$ {\rm (ii)}:
Suppose that $H_0(T)$ is closed and that there exists a closed
$T$-invariant subspace $M$ of $X$ such that $X=H_0(T) \oplus M$ and
$T(M)$ is closed. For $N=H_0(T)$  we have that  $(M,N)\in Red(T)$
and  $H_0(T_N)=N$, which implies that $T_N$ is quasinilpotent. From
$N(T_{M})=N(T)\cap M\subset H_0(T)\cap M=\{0\}$ it follows that
$T_{M}$ is injective and since $R(T_M)=T(M)$ is a closed  subspace
in $M$, we conclude that $T_M$ is bounded below.

{\rm (ii)} $\Longrightarrow$ {\rm (i)}: Assume  that there exists
$(M,N)\in Red(T)$ such that $ T_M$ is bounded below and $ T_N$ is
quasinilpotent. Then $(M,N)$ is a GKD for $T$, and so from
\cite[Corollary 1.69]{aiena2} it follows that $H_0(T)=
H_0(T_{M})\oplus H_0(T_{N})=H_0(T_{M})\oplus N$. Since $T_{M}$ is
bounded below,  we get that $H_0(T_{M})=\{0\}$ and hence $H_0(T)=N$.
Therefore, $H_0(T)$ is closed and complemented with $M$,
$(M,H_0(T))\in Red(T)$, and    $T(M)$ is closed because $T_{M}$ is
bounded below.

The implications {\rm (ii)} $\Longrightarrow$ {\rm (iii)} and {\rm
(vi)} $\Longrightarrow$ {\rm (vii)} can be proved analogously  to
the proof of the implication {\rm (i)} $\Longrightarrow$ {\rm (ii)}
in Theorem \ref{glavna}. The implications {\rm (iii)}
$\Longrightarrow$ {\rm (iv)} and {\rm (vii)} $\Longrightarrow$ {\rm
(viii)} are clear.

\smallskip

{\rm (viii)} $\Longrightarrow$ {\rm (ii)}:   Let $T$ admit a GKD and
let $0\notin\int\, \sigma_{{\bf \B}_+}(T)$, i.e. $0  \in \acc \,
\rho_{{\bf \B}_+}(T)$. There exists $(M,N) \in Red(T)$ such that
$T_M$ is Kato and $T_N$ is quasinilpotent.   From  $0  \in \acc \,
\rho_{{\bf \B}_+}(T)$ it follows that $0  \in \acc \, \rho_{{\bf
\B}_+}(T_M)$ according to Lemma ďż˝\ref{lema}{\rm (i)}.  From
Proposition ďż˝\ref{gap0}{\rm (iv)}  it follows that $T_M$ is
bounded below, and hence $T \in {\bf g D\M}(X)$.

{\rm (iv)} $\Longrightarrow$ {\rm (ii)}: This implication can be
proved by using Proposition \ref{gap0}(iii),  analogously to the
proof of the implication {\rm (viii)} $\Longrightarrow$ {\rm (ii)}.

{\rm (ii)} $\Longrightarrow$ {\rm (vi)}: Follows from the fact that
every bounded below operator is upper semi-Browder.

The equivalences {\rm (v)} $\Longleftrightarrow$ {\rm (ii)} and {\rm
(vi)} $\Longleftrightarrow$ {\rm (ix)} can be proved analogously to
the equivalence {\rm (i)} $\Longleftrightarrow$ {\rm (iv)}  in
Theorem \ref{glavna}.
\end{proof}

\begin{theorem} \label{t2}
For  $T\in L(X)$ the following conditions are equivalent:
\par \noindent {\rm (i)} $K(T)$ is closed
and there exists a closed subspace $N$ of $X$ such that $N \subset
H_0(T)$ and  $(K(T),N)\in Red(T)$;

\par \noindent {\rm
(ii)} There exists $(M,N)\in Red(T)$ such that $ T_M$ is surjective
and $ T_N$ is quasinilpotent, that is
 $T\in {\bf g D\mathcal{Q}}(X)$;


\par\noindent {\rm (iii)} $T$ admits a GKD and
$0\notin {\rm acc} \, \sigma_{su}(T)$;

\par\noindent {\rm (iv)} $T$ admits a GKD and
$0\notin{\rm int}\, \sigma_{su}(T)$;

\par \noindent {\rm (v)}
There exists a bounded projection $P$ on $X$ which commutes with $T$
such that $T+P$ is surjective and $TP$ is quasinilpotent;

\par \noindent {\rm
(vi)} There exists $(M,N)\in Red(T)$ such that $ T_M$ is lower
semi-Browder and $ T_N$ is quasinilpotent, that is $T\in {\bf g
D{\mathcal{B}_-}}(X)$;

\par \noindent {\rm (vii)} $T$ admits a GKD and
$0\notin {\rm acc} \, \sigma_{\mathcal{B}_-}(T)$;

\par \noindent {\rm (viii)}
$T$ admits a GKD and $0\notin{\rm int}\, \sigma_{\B_-}(T)$;

\par \noindent {\rm (ix)}
There exists a bounded projection $P$ on $X$ which commutes with $T$
such that $T+P$ is lower semi-Browder and $TP$ is quasinilpotent.

\par In particular, if $T$ satisfies any of the conditions {\rm (i)}-{\rm (ix)},
 then the subspace $M$ in {\rm (ii)} is uniquely determined and $M=K(T)$.
\end{theorem}
\begin{proof} {\rm (i)} $\Longrightarrow$ {\rm (ii)}:
Assume  that $K(T)$ is closed and that there exists a closed
$T$-invariant subspace $N$, such that  $N\subset H_0(T)$ and $X=K(T)
\oplus N$. For $M=K(T)$  we have that  $(M,N)\in Red(T)$,
$R(T_M)=R(T)\cap M=R(T)\cap K(T)=K(T)=M$, and so $T_M$ is
surjective. Since   $H_0(T_N)=H_0(T)\cap N=N$, we conclude that
$T_N$ is quasinilpotent.

{\rm (ii)} $\Longrightarrow$ {\rm (i)}: Suppose that there exists
$(M,N)\in Red(T)$ such that $ T_M$ is surjective and $ T_N$ is
quasinilpotent.  Then $(M,N)$ is a GKD for $T$ and from
\cite[Theorem 1.41]{aiena2} we obtain that $K(T)=K(T_M)$. Since
$T_M$ is surjective, it follows that $K(T_M)=M$, and so $K(T)=M$ and
 $K(T)$ is closed.  Thus $(K(T),N)\in Red(T)$ and since $ T_N$ is quasinilpotent, we have that $N=H_0(T_N)\subset H_0(T)$.

The rest of the proof is similar to the proofs of Theorems \ref{t1}
and \ref{glavna}.
\end{proof}

\noindent The equivalences {\rm (i)}$\Longleftrightarrow${\rm (ii)}
of Theorem ďż˝\ref{t1} and Theorem ďż˝\ref{t2} were proved
recently, see Proposition 3.2 and Proposition 3.4 in \cite{algeria},
but we emphasize that the authors permanently use term "left
invertible" instead of "bounded below" and "right invertible"
instead of "surjective". Also, the equivalences {\rm
(i)}$\Longleftrightarrow${\rm (v)} of the theorems mentioned above
are established in \cite{milos}.

\begin{remark} {\em If $T$ is generalized Drazin invertible, then from Theorem \ref{Kol-Dra} and Theorem \ref{t1} it follows that $T$
is generalized Drazin bounded below and in that case  the closed
subspace $M$ of $X$ which satisfies the condition (i) in Theorem
\ref{t1}, i.e. such that $(M,H_0(T))\in Red(T)$ and $T(M)$ is
closed, is uniquely determined-we show that it must be equal to
$K(T)$. In other words, the projection $P$ which satisfies the
condition (v) in Theorem \ref{t1} is uniquely determined-it is equal
to the spectral idempotent of $T$ corresponding to the set $\{0\}$.

Indeed, from Theorem \ref{t1} it follows that $T=T_M\oplus
T_{H_0(T)}$, $T_M$ is bounded below and $T_{H_0(T)}$ is
quasinilpotent. Since $T$ is generalized Drazin invertible, we have
that $0\notin\acc\, \sigma(T)$,  and hence, $0\notin\acc\,
\sigma(T_M)$. $T_M$ is Kato since it is bounded below, and so by
Proposition \ref{gap0}{\rm (i)} we obtain that $T_M$ is invertible.
Since $T$ admits a GKD $(M,H_0(T))$, from \cite[Theorem
3.15]{aiena2} it follows that $M=K(T)$. The similar observation can
be stated in the context of Theorem ˜\ref{t2}.}
\end{remark}

In the following theorem we give several necessary and sufficient
conditions for $T\in L(X)$ to be generalized Drazin invertible.

\begin{theorem}\label{t3}
Let $T\in L(X)$. The following conditions are equivalent:

\snoi {\rm (i)} $T$ is generalized Drazin invertible;

\snoi {\rm (ii)} $T$ admits a GKD and $0\notin{\rm int}\,
\sigma(T)$;

\snoi {\rm (iii)} $T$ admits a GKD and $0\notin{\rm int}\,
\sigma_{\B}(T)$;

\snoi {\rm (iv)} $T$ admits a GKD and $0\notin {\rm acc} \,
\sigma_{\mathcal{B}}(T)$;

\snoi {\rm (v)} There exists $(M,N)\in Red(T)$ such that $ T_M$ is
Browder and $ T_N$ is quasinilpotent;

\snoi {\rm (vi)} There exists a bounded projection $P$ on $X$ which
commutes with $T$ such that $T+P$ is Browder and $TP$ is
quasinilpotent.
\end{theorem}
\begin{proof} Similar to the proof of Theorem \ref{t1}.
\end{proof}


Analysis similar to that in the proof of Theorem ˜\ref{glavna} gives
the following result.

\begin{theorem}\label{B-Fredholm}
Let $T \in L(X)$ and $1 \leq i \leq 12$. The following conditions
are equivalent:\par \noindent {\rm (i)} There exists $(M,N)\in
Red(T)$ such that $ T_M\in {\bf R}_i$  and $ T_N$ is nilpotent, that
is $T \in {\bf DR}_i$;\par \noindent {\rm (ii)} $T$ is of Kato type
and $0 \not \in \acc \, \sigma_{{\bf R}_i}(T)$;
\par\noindent {\rm (iii)} $T$ is of Kato type and
$0  \notin \int \, \sigma_{{\bf R}_i}(T)$;

\par \noindent {\rm
(iv)} There exists a projection $P \in L(X)$ that commutes with $T$
such that $T+P \in {\bf R}_i$ and $TP$ is nilpotent.
\end{theorem}

\noindent Using \cite[Theorem 2.7]{Ber} and \cite[Lemma 4.1]{Ber2}
we see that if $i=3$ $(i=6)$ then the conditions {\rm (i)}-{\rm
(iv)} of Theorem ˜\ref{B-Fredholm} are equivalent to the fact that
$T$ is B-Fredholm ($T$ is B-Weyl), while if $i=12$ these conditions
are equivalent to the fact that $T$ is Drazin invertible. Similar to
the definitions of the B-Fredholm and B-Weyl operators, the classes
${\bf BR}_i$ are introduced and studied \cite{Ber0}. If $i \in
\{1,2,4,5,7,8,10,11\}$ then it is possible to show that the
conditions {\rm (i)}-{\rm (iv)} of Theorem ˜\ref{B-Fredholm} are
equivalent to the fact that $T$ is of Kato type and $T$ belongs to
the class ${\bf BR}_i$. For the case of a Hilbert space see
\cite[Theorem 3.12]{Ber0}.

\begin{corollary} \label{t4-GKt} Let $T\in L(X)$ and let $0\in\partial\sigma_{{\bf R}_i}(T), \, 1 \leq i \leq
12$. Then the following assertions hold.\par

\noindent {\rm (i)} $T$ admits a generalized Kato decomposition if
and only if $T$ belongs to the class ${\bf gDR}_i$. \par

\noindent {\rm (ii)} $T$ is of Kato type if and only if $T$ belongs
to the class ${\bf DR}_i$.
\end{corollary}
\begin{proof}  Follows from the equivalence {\rm (i)}$\Longleftrightarrow${\rm (iii)} in Theorem \ref{glavna}, the equivalences
{\rm (ii)}$\Longleftrightarrow${\rm (iv)} in Theorems \ref{t1} and
\ref{t2}, the equivalence {\rm (i)}$\Longleftrightarrow${\rm (ii)}
in Theorem \ref{t3} and from the equivalence {\rm
(i)}$\Longleftrightarrow${\rm (iii)} in Theorem ˜\ref{B-Fredholm}.
\end{proof}

\noindent The case $i=12$ is proved in \cite[Theorem
2.9]{aienarosas} and \cite[Theorem 3.8]{kinezi}.

\begin{remark}\label{poi} {\em We remark that
\begin{eqnarray*}
   &\Phi_+(X)\setminus\W_+(X)\subset {\bf gD\Phi_+}(X)\setminus {\bf gD\W_+}(X),\label{gh0}\\
   &\Phi_-(X)\setminus\W_-(X)\subset {\bf gD\Phi_-}(X)\setminus {\bf gD\W_-}(X),\label{gh}\\
   &\Phi(X)\setminus\W(X)\subset {\bf gD\Phi}(X)\setminus {\bf gD\W}(X).\label{gh00}
\end{eqnarray*}
Indeed, as the index is locally constant, the set
$\Phi_+(X)\setminus\W_+(X)=\{T\in \Phi(X): \ind(T)>0\}$ is open,
which implies that the set
$\sigma_{\W_+}(T)\setminus\sigma_{\Phi_+}(T)=\rho_{\Phi_+}(T)\setminus\rho_{\W_+}(T)$
is open for every $T\in L(X)$. Suppose that $T\in
\Phi_+(X)\setminus\W_+(X)$. Then $T\in {\bf gD \Phi_+}(X)$ and $0\in
\sigma_{\W_+}(T)\setminus\sigma_{\Phi_+}(T)$. There exists
$\epsilon>0$ such that $K(0,\epsilon)\subset
\sigma_{\W_+}(T)\setminus\sigma_{\Phi_+}(T)$. Hence, $0\in \acc\,
\sigma_{\W_+}(T)$ and $T\notin {\bf gD\W_+}(X)$  according to
Theorem \ref{glavna}. Similarly for the remaining inclusions.}
\end{remark}
\noindent The next example shows that the inclusions ${\bf
gD\W_+}(X)\subset {\bf gD\Phi_+}(X)$, ${\bf gD\W_-}(X)  \\ \subset
{\bf gD\Phi_-}(X)$  and ${\bf gD\W}(X)\subset {\bf gD\Phi}(X)$ can
be proper.
\begin{example} \label{ex2}
{\em Let $U$ and $V$ be the forward and the backward unilateral
shifts defined on $\ell_p(\NN), \, (1\leq p < \infty)$,
respectively. The operators $U$ and $V$ are Fredholm, $\ind(U)=-1$
and $\ind(V)=1$. Therefore, $U\in \Phi_-(X)\setminus\W_-(X)$ and
$V\in \Phi_+(X)\setminus\W_+(X)$, and also $U,V\in
\Phi(X)\setminus\W(X)$. Hence, according to Remark \ref{poi}, $U\in
{\bf gD\Phi_-}(X)\setminus {\bf gD\W_-}(X)$, $V\in {\bf
gD\Phi_+}(X)\setminus {\bf gD\W_+}(X)$ and $U,V\in {\bf
gD\Phi}(X)\setminus {\bf gD\W}(X)$.}
\end{example}
\noindent We also show that the inclusions ${\bf gD\M}(X)\subset
{\bf gD\W_+}(X)$ and ${\bf gD\Q}(X)\subset {\bf gD\W_-}(X)$ can be
proper.
\begin{example} \label{ex3}
{\em Let $U$ and $V$ be as in Example ˜\ref{ex2} and let $T=U\oplus
V$. Then, according to Lemma \ref{lema}(i), $T$ is Fredholm and
$\ind(T)=\ind(U)+\ind(V)=0$. Thus $T$ is Weyl and hence, $T$ is
generalized Drazin Weyl. Since $\sigma_{ap}(U)=
\sigma_{su}(V)=\partial\DD$ and $\sigma_{su}(U)=
\sigma_{ap}(V)=\DD$, where $\mathbb{D}=\{\lambda \in
\mathbb{C}:|\lambda|\leq1\}$, it follows that
$\sigma_{ap}(T)=\sigma_{ap}(U)\cup \sigma_{ap}(V)=\DD$ and
$\sigma_{su}(T)=\sigma_{su}(U)\cup \sigma_{su}(V)=\DD$. Therefore,
$0\in\acc\, \sigma_{ap}(T)$ and $0\in\acc\, \sigma_{su}(T)$ and from
Theorems \ref{t1} and \ref{t2} it follows that $T$ is neither
generalized Drazin bounded below nor generalized Drazin surjective.}
\end{example}

\section{Applications}

For $T\in L(X)$ we define the spectra with respect to the sets ${\bf
g D R}_i, \; 1\leq i \leq 12$, in a classical way,

\[\sigma_{ {\bf g D R}_i}(T)=\{\lambda \in \mathbb{C}:T-\lambda \not \in
{\bf g D R}_i\}, \; \; 1\leq i \leq 12.\] \noindent From Theorems
\ref{glavna}, \ref{t1} and \ref{t2} it follows that
\begin{eqnarray}
 \sigma_{{ \bf g D R}_i}(T)&=&\sigma_{gK}(T) \cup \acc \, \sigma_{{\bf
R}_i}(T)\nonumber\\&=&\sigma_{gK}(T) \cup \int\, \sigma_{{\bf
R}_i}(T), \ 1\leq i \leq 12.\label{glava-}
\end{eqnarray}

\noindent From Theorem ˜\ref{B-Fredholm} we conclude
\begin{equation} \label{inkluzijeBF}
\sigma_{B\Phi}(T)=\sigma_{Kt}(T) \cup \int \, \sigma_{\Phi}(T) \; \;
\text{and} \; \; \sigma_{B\W}(T)=\sigma_{Kt}(T) \cup \int \,
\sigma_{\W}(T).
\end{equation}

\noindent From \eqref{glava-} it follows that if $T\in L(X)$ has the
property that
\begin{equation*}
\sigma_{{\bf R}_i}(T)=\partial\sigma_{{\bf R}_i}(T),
\end{equation*}
then
\[
  \sigma_{gK}(T)=\sigma_{{\bf gDR}_i}(T),\ 1 \leq i \leq 12.
\]
Consequently,  if $\sigma(T)$ is at most countable or  contained in
a line, then $\sigma_{gK}(T)=\sigma_{gD}(T)$. Every  self-adjoint,
as well as, unitary operator on a Hilbert space have the spectrum
contained in a line.  The spectrum of polynomially meromorphic
operator \cite{KaashoekPRIA} is at most countable.

\begin{proposition} \label{cor1}
Let $T \in L(X)$ and $1\leq i \leq 12$. The following statements
hold:
\par \noindent {\rm (i)} $\sigma_{{\bf g D R}_i}(T) \subset
\sigma_{{\bf R}_i}(T) \subset \sigma(T)$;\par \noindent {\rm (ii)}
$\sigma_{{\bf g D R}_i}(T)$ is a compact subset of $\mathbb{C}$;\par
\noindent {\rm (iii)} $\sigma_{{\bf R}_i}(T) \setminus \sigma_{{\bf
g D R}_i}(T)$ consists  of at most countably many isolated points.
\end{proposition}
\begin{proof}
{\rm (i):} It is obvious.\par \noindent {\rm (ii):} It suffices to
prove that $\sigma_{{\bf g D R}_i}(T)$ is closed since it is bounded
by the part {\rm (i)}. If $\lambda_0 \not \in \sigma_{{\bf g D
R}_i}(T)$, then $T-\lambda_0 \in {\bf g D R}_i$ and by Proposition
ďż˝\ref{open} there exists $\epsilon>0$ such that
$T-\lambda_0-\lambda \in {\bf R}_i\subset {\bf g D R}_i$ for
$0<|\lambda|<\epsilon$. It means that open disc of radius $\epsilon$
centered at $\lambda_0$ is contained in the complement of
$\sigma_{{\bf g D R}_i}(T)$, so $\sigma_{{\bf g D R}_i}(T)$ is
closed.
\par \noindent {\rm (iii):} If $\lambda\in \sigma_{{\bf R}_i}(T) \setminus \sigma_{{\bf
g D R}_i}(T)$, then $\lambda \in \sigma_{{\bf R}_i}(T)$ and
$T-\lambda\in {\bf g D R}_i$. Applying Proposition ďż˝\ref{open}
we obtain that  $\lambda \in  \iso \, \sigma_{{\bf R}_i}(T)$, and
hence $\sigma_{{\bf R}_i}(T) \setminus \sigma_{{\bf g D R}_i}(T)$
consists  of at most countably many isolated points.
\end{proof}

\begin{corollary} \label{cor5}
Let $T \in L(X)$. Then the following inclusions hold:
\begin{eqnarray*}
  \acc \, \sigma_{ap}(T) \setminus \acc \, \sigma_{\mathcal{B}_+}(T) &\subset& \sigma_{gK}(T), \\
  \acc \, \sigma_{su}(T) \setminus \acc \, \sigma_{\mathcal{B}_-}(T) &\subset& \sigma_{gK}(T), \\
  \acc \, \sigma(T) \setminus \acc \, \sigma_{\mathcal{B}}(T) &\subset& \sigma_{gK}(T),\\
\int \, \sigma_{ap}(T) \setminus \int \, \sigma_{\mathcal{B}_+}(T) &\subset& \sigma_{gK}(T), \\
  \int \, \sigma_{su}(T) \setminus \int \, \sigma_{\mathcal{B}_-}(T) &\subset& \sigma_{gK}(T), \\
  \int \, \sigma(T) \setminus \int \, \sigma_{\mathcal{B}}(T) &\subset& \sigma_{gK}(T).
\end{eqnarray*}
 \end{corollary}
 \begin{proof}
Follows from the equivalences {\rm (iii)} $\Longleftrightarrow  $
{\rm (vii)} and {\rm (iv)} $\Longleftrightarrow  $ {\rm (viii)} in
Theorems \ref{t1} and \ref{t2}.
\end{proof}

\begin{remark}
{\em Let $T\in L(X)$ be a Riesz operator with infinite spectrum. It
was shown in \cite{kinezi} that $T$ does not admit a GKD. It is
interesting to note that the same follows from Corollary
ďż˝\ref{cor5}. Namely, $\sigma_{\mathcal{B}}(T)=\{0\}$ and so $0
\not \in \acc \, \sigma_{\B}(T)$, while  $0\in \acc \, \sigma(T) $.
Therefore, $0\in \acc \, \sigma(T)\setminus \acc \, \sigma_{\B}(T)$
and hence $0\in \sigma_{gK}(T)$ by Corollary \ref{cor5}.


According to Theorems \ref{t1} and  \ref{t2} it follows that $T$ is not generalized Drazin bounded below (surjective), while according to Theorem \ref{glavna} 
it follows that $T$ is not generalized Drazin upper (lower)
semi-Fredholm. From $\sigma(T)=\sigma_{ap}(T)=\sigma_{su}(T)$ and
$\sigma_{{\bf R}_i}(T) =\{0\}$, $1\le i\le 9$,  it follows that
$0\notin \int\,  \sigma_{{\bf R}_i}(T) =\emptyset$, $1\le i\le 12$.
This shows that the condition that the operator admits a GKD in the
statements (iv), (vii) and (viii) of Theorems \ref{t1} and \ref{t2},
as well as in the statements (ii), (iii) and (iv) of Theorem
\ref{t3} and also, in the statements (ii) and (iii) of Theorem
\ref{glavna}, can not be omitted.}
\end{remark}
\noindent The following question is natural.

\begin{question}
{\em Does an operator which does not admit a GKD and such that $0$
is not an accumulation point of its approximate point (resp.
surjective) spectrum exist? If the answer is affirmative, then it
means that the condition that $T$ admits a GKD in the statements
(iii) of Theorems \ref{t1} and \ref{t2} can not be omitted.}
\end{question}

K. Miloud Hocine et al. \cite[Theorems 3.16, 3.17]{algeria}
mentioned that the sufficient condition for $T$ to be a sum of a
bounded below (surjective) operator and a quasinilpotent one, is
that  $0\in \iso\, \sigma_{ap}(T)$ ($0\in \iso\, \sigma_{su}(T)$),
but they did not give the proof. In  \cite[Propositions 3.5,
3.7]{{Benharrat-novi}} the authors proved that the condition that
$0\in \iso\, \sigma_{ap}(T)$ ($0\in \iso\, \sigma_{su}(T)$) implies
that $T$ admits a GKD, but their proof is incorrect. Namely, from
the assumption  $0\in\iso\, \sigma_{ap}(T)$ they conclude wrongly
that $T$ is Kato (by incorrect using \cite[Th\`{e}or\'{e}me
2.1]{Bou}), and from this fact they obtain that $T$ is bounded
below,
 what  is  a contradiction-this
would mean that for  arbitrary bounded  linear operator $T$, $\iso\,
\sigma_{ap}(T)=\emptyset$, which, of course, is not true.

\medskip

We give an alternative proof of the inclusion
$$
\partial\sigma(T)\cap\acc\, \sigma(T)\subset\sigma_{gK}(T)
$$
from  Q. Jiang and H. Zhong's paper \cite{kinezi} (see Corollary 3.6
and Theorem 3.8), which does not require any use of the
single-valued extension property. What is more, we establish the
inclusions of the same type for other spectra.

\begin{theorem}\label{partial} Let $T\in L(X)$. Then the following inclusions hold:
\begin{equation}\label{poss-0}
 \partial\sigma_{{\bf R}_i}(T)\cap\acc\, \sigma_{{\bf R}_i}(T)\subset \sigma_{gK}(T) ,\ \ 1\le  i\le 12.
\end{equation}
Moreover,
\begin{eqnarray*}
  \partial\sigma_{\B_+}(T)\cap\acc\, \sigma_{ap}(T) &\subset& \sigma_{gK}(T);\label{poss-1} \\
   \partial\sigma_{\B_-}(T)\cap\acc\, \sigma_{su}(T) &\subset& \sigma_{gK}(T);\label{poss-2} \\
    \partial\sigma_{\B}(T)\cap\acc\, \sigma(T) &\subset& \sigma_{gK}(T).\label{poss-3}
\end{eqnarray*}
\end{theorem}
\begin{proof}

From Theorems \ref{t3} and \ref{Kol-Dra},  the equivalence
(iii)$\Longleftrightarrow$(iv) in  Theorems \ref{t1} and  \ref{t2}
and  the equivalence (ii)$\Longleftrightarrow$(iii) in  Theorem
\ref{glavna} it foollows that if $T-\lambda \in L(X)$ admits a GKD,
then
$$
\lambda\in\int\; \sigma_{{\bf R}_i}(T)\Longleftrightarrow
\lambda\in\acc\; \sigma_{{\bf R}_i}(T),\ 1\leq i \leq 12.
$$
Therefore, we have the inclusions $$\partial\sigma_{{\bf
R}_i}(T)\cap \acc\, \sigma_{{\bf R}_i}(T)= \acc\, \sigma_{{\bf
R}_i}(T)\setminus \int\, \sigma_{{\bf
R}_i}(T)\subset\sigma_{gK}(T),\ 1\leq i \leq 12.$$

Suppose that $\lambda\in\partial\sigma_{\B_+}(T)$ and $T-\lambda$
admits a GKD. Then $\lambda\notin\int\; \sigma_{\B_+}(T)$ and from
the equivalence {\rm (viii)}$\Longleftrightarrow${\rm (iii)} in
Theorem \ref{t1} we get that $\lambda\notin\acc\, \sigma_{ap}(T)$.
Therefore, if $\lambda\in\partial\sigma_{\B_+}(T)\cap \acc\,
\sigma_{ap}(T)$, then $T-\lambda$ does not admit a GKD, i.e.
$\lambda\in  \sigma_{gK}(T)$. The remaining inclusions can be proved
analogously.
\end{proof}

\noindent From \eqref{poss-0} it follows that

\begin{equation*}
 \partial\sigma_{{\bf R}_i}(T)\subset \sigma_{gK}(T)\cup{\iso}\, \sigma_{{\bf R}_i}(T),\ \ 1\le  i\le 12,
\end{equation*}
which implies that
$$
\partial\sigma_{{\bf R}_i}(T)\setminus \sigma_{gK}(T)\subset {\iso}\, \sigma_{{\bf R}_i}(T),
 $$
and hence $\partial\sigma_{{\bf R}_i}(T)\setminus \sigma_{gK}(T)$
consist of at most countably many points.


\bigskip

The {\it connected hull}  of a compact subset $K$ of the complex
plane $\CC$, denoted by $\eta K$, is the complement of the unbounded
component of $\CC\setminus K$ (\cite[Definition 7.10.1]{H3}). Given
a compact subset $K$ of the plane, a hole of $K$ is a bounded
component of $\CC\setminus K$, and so a hole of $K$ is a component
of $\eta K\setminus K$. Generally (\cite[Theorem 7.10.3]{H3}), for
compact subsets $H,K\subset\CC$,
\begin{equation}\label{spec.1}
\partial H\subset K\subset
H\Longrightarrow\partial H\subset\partial K\subset K\subset
H\subset\eta K=\eta H\ .
\end{equation}
Evidently, if $K\subseteq\CC$ is finite, then $\eta K=K$. Therefore,
for compact subsets $H,K\subseteq\CC$,
\begin{equation}\label{spec.2}
 \eta K=\eta H\Longrightarrow(H\ {\rm is\
finite}\Longleftrightarrow K\ {\rm is\ finite}),
\end{equation}
and in that case $H=K$.

\medskip

For $\lambda_0\in \CC$ we set $K(\lambda_0,\epsilon)=\{\lambda\in
\CC: |\lambda-\lambda_0|<\epsilon\}$.

\begin{theorem}\label{rub} Let $T\in L(X)$. Then

%
\begin{center}
\begin{tabular}{rcccl}
& & $\partial \sigma_{gD\mathcal{M}}(T) \; \; \subset \; \; \partial \sigma_{gD\mathcal{W}_+}(T) \; \; \subset \; \; \partial \sigma_{gD\Phi_+}(T)$& & \\
 & \rotatebox{30}{$\subset$} & &\rotatebox{-30}{$\subset$}& \\
$\partial \sigma_{gD}(T)$ & $\subset$ &$\partial \sigma_{gD\mathcal{W}}(T) \; \; \; \; \; \; \; \subset \; \; \; \; \; \; \; \partial \sigma_{gD\Phi}(T)$& $\subset$ &$\partial \sigma_{gK}(T)$\\
 &\rotatebox{-30}{$\subset$} & & \rotatebox{30}{$\subset$} &\\
& & $\partial \sigma_{gD\mathcal{Q}}(T) \; \; \subset \; \; \partial \sigma_{gD\mathcal{W}_-}(T) \; \; \subset \; \; \partial \sigma_{gD\Phi_-}(T)$ & &\\

\end{tabular}
\end{center}

\begin{eqnarray*}
&\partial\sigma_{gD \Phi}(T)\subset \partial\sigma_{gD \Phi_+}(T),\ \  \partial\sigma_{gD \Phi}(T)\subset \partial\sigma_{gD \Phi_-}(T), \\
  &\partial\sigma_{gD \W}(T)\subset \partial\sigma_{gD \W_+}(T),\ \  \partial\sigma_{gD \W}(T)\subset \partial\sigma_{gD \W_-}(T),
  \end{eqnarray*}


and

\begin{eqnarray}%
\eta\sigma_{gK}(T)&=&\eta\sigma_{gD\Phi_+}(T)=\eta\sigma_{gD\W_+}(T)=\eta\sigma_{gD\M}(T)\nonumber\\&=&\eta\sigma_{gD\Phi_-}(T)=\eta\sigma_{gD\W_-}(T)=\eta\sigma_{gD\Q}(T)\label{rub2}
\\&=&\eta\sigma_{gD\Phi}(T)=\eta\sigma_{gD\W}(T)=\eta\sigma_{gD}(T).\nonumber
\end{eqnarray}
\end{theorem}
\begin{proof} According to (\ref{spec.1})
 it is sufficient to prove the inclusions

\snoi {\rm (i)} $ \partial\sigma_{gD}(T)\subset \sigma_{gK}(T)$;
\quad  {\rm (ii)} $ \partial\sigma_{gD\M}(T)\subset \sigma_{gK}(T)$;
\quad {\rm (iii)} $ \partial\sigma_{gD\W_+}(T)\subset
\sigma_{gK}(T)$;

\snoi {\rm (iv)} $ \partial\sigma_{gD\Phi_+}(T)\subset
\sigma_{gK}(T)$; \ \ {\rm (v)} $ \partial\sigma_{gD\Q}(T)\subset
\sigma_{gK}(T)$; \  {\rm (vi)} $ \partial\sigma_{gD\W_-}(T)\subset
\sigma_{gK}(T)$;

\snoi {\rm (vii)} $ \partial\sigma_{gD\Phi_-}(T)\subset
\sigma_{gK}(T)$; \ {\rm (viii)} $ \partial\sigma_{gD\W}(T)\subset
\sigma_{gK}(T)$; \ {\rm (ix)} $ \partial\sigma_{gD\Phi}(T)\subset
\sigma_{gK}(T)$.

\medskip

 Suppose that $\lambda_0\in \partial\sigma_{gD}(T)$. Since $\sigma_{gD}(T)$ is closed, it follows that
  \begin{equation}\label{mlad}
    \lambda_0\in \sigma_{gD}(T)=\sigma_{gK}(T)\cup{\rm int}\, \sigma(T).
  \end{equation}

 We prove that
 \begin{equation}\label{mlad1}
   \lambda_0\notin {\rm int}\, \sigma(T).
\end{equation}
Suppose  on the contrary that $\lambda_0\in {\rm int}\, \sigma(T)$.
Then there exists an $\epsilon>0$ such that
$K(\lambda_0,\epsilon)\subset \sigma(T)$. This means that
$K(\lambda_0,\epsilon)\subset {\rm int}\,\sigma(T)$ and hence,
$K(\lambda_0,\epsilon)\subset \sigma_{gD}(T)$, which contradicts the
fact that $\lambda_0\in \partial\sigma_{gD}(T)$. Now from
(\ref{mlad}) and (\ref{mlad1}), it follows that
$\lambda_0\in\sigma_{gK}(T)$.

The inclusions (ii)-(ix) can be proved similarly to the inclusion
(i).
\end{proof}

\smallskip

\noindent From (\ref{spec.2}) and (\ref{rub2}) it follows that
$\sigma_{gK}(T)$ is finite if and only if $\sigma_{{\bf g D
R}_i}(T)$ is finite for any $i\in\{1,\dots ,12\}$ and in that case
$\sigma_{gK}(T)=\sigma_{{\bf g D R}_i}(T)$ for all $i\in\{1,\dots
,12\}$. In particular, $\sigma_{gK}(T)\ {\rm  is\ finite}$ if and
only if $\sigma_{gD}(T)\ {\rm  is\  finite},$
 and in that case $\sigma_{gK}(T)=\sigma_{gD}(T)$. Moreover, $\sigma_{gK}(T)=\emptyset$ if and only if $\sigma_{{\bf g D R}_i}(T)=\emptyset$  for any
$i\in\{1,\dots ,12\}$, that is, if and only if  $\sigma (T)$ is a
finite set.

\smallskip

The Drazin spectrum of $T \in L(X)$ is defined as
\[
\sigma_{D}(T)=\{\lambda \in \CC: T-\lambda \; \text{is not Drazin
invertible}\}.
\]

\begin{theorem}\label{rub-Kt} Let $T\in L(X)$. Then

\[\partial \sigma_{D}(T) \subset \partial \sigma_{B\mathcal{W}}(T) \subset \partial \sigma_{B\Phi}(T) \subset \partial
\sigma_{Kt}(T)\]

and

\[
\eta\sigma_{Kt}(T)=\eta\sigma_{B\Phi}(T)=\eta\sigma_{B\W}(T)=\eta\sigma_{D}(T).\nonumber
\]
\end{theorem}
\begin{proof} From \cite[Theorem 2.9]{aienarosas} we have  $\sigma_D(T)=\sigma_{Kt}(T) \cup \int \, \sigma(T)$. Now use this equality and
˜\eqref{inkluzijeBF}, and proceed similarly as in the proof of
Theorem \ref{rub}.
\end{proof}

\smallskip

The generalized Drazin resolvent set of $T\in L(X)$ is defined by
$\rho_{gK}(T)=\CC\setminus\sigma_{gK}(T)$. As a consequence of
Theorem \ref{rub} we get Theorem 3 in \cite{kinezi2012}, which was
proved by using the technique based on the single-valued extension
property.

\begin{corollary}[\cite{kinezi2012}, Theorem 3] \label{componenta} Let $T\in L(X)$ and let $\rho_{gK}(T)$ has only one component. Then
$$
\sigma_{gK}(T)=\sigma_{gD}(T).
$$
\end{corollary}
\begin{proof} Since
$\rho_{gK}(T)$ has only one component, it follows that
$\sigma_{gK}(T)$ has no holes, and so  $\sigma_{gK}(T)=\eta
\sigma_{gK}(T)$. From (\ref{rub2}) it follows that
$\sigma_{gD}(T)\supset \sigma_{gK}(T)=\eta \sigma_{gK}(T)= \eta
\sigma_{gD}(T)\supset \sigma_{gD}(T)$ and hence
$\sigma_{gD}(T)=\sigma_{gK}(T)$.
\end{proof}


\begin{theorem}\label{nov} Let $T\in L(X)$ and $1\le i\le 12$.
If
\begin{equation}\label{bon}
 \partial\sigma_{{\bf R}_i}(T)\subset\acc\, \sigma_{{\bf R}_i}(T),
\end{equation}
then
%
\begin{equation}\label{part1}
 \partial\sigma_{{\bf R}_i}(T)\subset \sigma_{gK}(T)\subset \sigma_{Kt}(T)\subset \sigma_{eK}(T)\subset \sigma_{{\bf R}_i}(T)
 \end{equation}
and
\begin{equation}\label{pacc1}
  \eta\sigma_{{\bf R}_i}(T)=\eta \sigma_{gK}(T)=\eta \sigma_{Kt}(T)=\eta \sigma_{eK}(T).
\end{equation}
\end{theorem}
\begin{proof}
From $\partial\sigma_{{\bf R}_i}(T)\subset\acc\, \sigma_{{\bf R}_i}(T)$ it follows that $\partial\sigma_{{\bf R}_i}(T)\cap\acc\, \sigma_{{\bf R}_i}(T)=\partial\sigma_{{\bf R}_i}(T)$, and so  from (\ref{poss-0}) it follows that $\partial\sigma_{{\bf R}_i}(T)\subset \sigma_{gK}(T)$. 
(\ref{pacc1}) follows from (\ref{part1}) and (\ref{spec.1}).
\end{proof}

The Goldberg spectrum of $T\in L(X)$ is defined by
$$\sigma_{ec}(T)=\{\lambda\in\CC:R(T-\lambda)\ {\rm is\ not\ closed}\}.
$$
Obviously, $\sigma_{ec}(T)\subset\sigma_{{\bf R}_i}(T)$ for $1 \leq
i \leq 12$.
\begin{theorem}\label{Ben}  Let $T\in L(X)$ and $1\le i\le 12$. If
\begin{equation}\label{bac}
\sigma_{{\bf R}_i}(T)=\partial\sigma_{{\bf R}_i}(T)=\acc\,
\sigma_{{\bf R}_i}(T),
\end{equation}
then
\begin{equation}\label{sp}
\sigma_{ec}(T)\subset
\sigma_{gK}(T)=\sigma_{Kt}(T)=\sigma_{eK}(T)=\sigma_{{\bf
R}_i}(T)=\sigma_{{\bf g D R}_i}(T).
\end{equation}
\end{theorem}
\begin{proof} Suppose that $\sigma_{{\bf R}_i}(T)=\partial\sigma_{{\bf R}_i}(T)$ and that every $\lambda\in  \sigma_{{\bf R}_i}(T)$ is not isolated in $\sigma_{{\bf R}_i}(T)$.
From Theorem \ref{nov} it follows that
\begin{eqnarray*}
  \sigma_{{\bf R}_i}(T)=\partial\sigma_{{\bf R}_i}(T)\subset \sigma_{gK}(T)\subset \sigma_{Kt}(T)\subset \sigma_{eK}(T)\subset \sigma_{{\bf R}_i}(T),
\end{eqnarray*}
and so
$$
\sigma_{ec}(T)\subset\sigma_{{\bf
R}_i}(T)=\sigma_{gK}(T)=\sigma_{Kt}(T)=\sigma_{eK}(T)= \sigma_{{\bf
g D R}_i}(T).
$$
\end{proof}
\begin{example}{\em Let $U$ and $V$ be as in Example ˜\ref{ex2}. Since $\sigma_{ap}(U)=\sigma_{su}(V)=\partial\DD,$
then $\sigma_{ap}(U)=\partial \sigma_{ap}(U)=\acc\, \sigma_{ap}(U)$
and $\sigma_{su}(V)=\partial\sigma_{su}(V)=\acc\, \sigma_{su}(V)$.
From Theorem \ref{Ben}  we get
\begin{eqnarray*}
    \sigma_{gK}(U)=\sigma_{K}(U)=\sigma_{ap}(U)=\sigma_{g D\M}(U)=\sigma_{g D\W_+}(U)=\sigma_{g D\Phi_+}(U)=\partial\DD,\\
     \sigma_{gK}(V)=\sigma_{K}(V)=\sigma_{su}(V)=\sigma_{g D\Q}(V)=\sigma_{g D\W_-}(V)=\sigma_{g D\Phi_-}(V)=\partial\DD.
\end{eqnarray*}}
\end{example}

It is not difficult to see that if $K\subset\CC$ is compact, then
for $\lambda\in\partial K$  there is equivalence: $ \lambda\in
\iso\, K\Longleftrightarrow \lambda\in \iso\, \partial K$, that is,
\begin{equation}\label{zz}
 \lambda\in \acc\, K\Longleftrightarrow
\lambda\in \acc\, \partial K.
\end{equation}
Moreover, it is well known that
\begin{eqnarray}
  &\partial\sigma(T)\subset\partial\sigma_{ap}(T)\subset\sigma_{ap}(T),\label{ink1}\\
  &\partial\sigma(T)\subset\partial\sigma_{su}(T)\subset\sigma_{su}(T).\label{ink2}
\end{eqnarray}
Now, as  consequences we get
 Theorem 3.12 and Corollary 3.13 in \cite{kinezi}, as well as Theorems 3.8 and 3.9 in \cite{BenBug}.
\begin{corollary}[\cite{kinezi},\cite{BenBug}] \label{PO} Let $T\in L(X)$ be an operator for which $\sigma_{ap}(T)=\partial\sigma (T)$ and every $\lambda\in \partial\sigma (T)$ is not isolated in $\sigma (T)$. Then
\begin{equation}\label{z}
  \sigma_{ec}(T)\subset\sigma_{ap}(T)=\sigma_{gK}(T)=\sigma_{Kt}(T)=\sigma_{eK}(T)=\sigma_{K}(T).
\end{equation}
\end{corollary}
\begin{proof}
From $\sigma_{ap}(T)=\partial\sigma (T)$ and (\ref{ink1}) it follows
that $\sigma_{ap}(T)=\partial\sigma_{ap}(T)$, while from (\ref{zz})
it follows that every $\lambda\in \partial\sigma (T)$ is not
isolated in $\partial\sigma (T)$, i.e. in $\sigma_{ap}(T)$. Now from
Theorem \ref{Ben}  we get (\ref{z}).
\end{proof}
\begin{corollary}[\cite{kinezi},\cite{BenBug}] Let $T\in L(X)$ be an operator for which $\sigma_{su}(T)=\partial\sigma (T)$ and every $\lambda\in \partial\sigma (T)$ is not isolated in $\sigma (T)$. Then
$
  \sigma_{ec}(T)\subset\sigma_{su}(T)=\sigma_{gK}(T)=\sigma_{Kt}(T)=\sigma_{eK}(T)=\sigma_{K}(T).$
\end{corollary}
\begin{proof}
Follows from (\ref{ink2}), (\ref{zz})  and  Theorem \ref{Ben},
analogously to the proof of Corollary \ref{PO}.
\end{proof}

\medskip\noindent
\author{Milo\v s D. Cvetkovi\'c}

\noindent University of Ni\v s\\
Faculty of Sciences and Mathematics\\
P.O. Box 224, 18000 Ni\v s, Serbia

\noindent {\it E-mail}: {\tt miloscvetkovic83@gmail.com}

\bigskip \noindent
\author{Sne\v zana
\v C. \v Zivkovi\'c-Zlatanovi\'c}

\noindent{University of Ni\v s\\
Faculty of Sciences and Mathematics\\
P.O. Box 224, 18000 Ni\v s, Serbia}

\noindent {\it E-mail}: {\tt mladvlad@mts.rs}

\end{document}